\documentclass[12pt]{amsart}
\usepackage{amsmath,amssymb,amsbsy,amsfonts,latexsym,amsopn,amstext,
                                               amsxtra,euscript,amscd,bm}
\usepackage[margin=1in]{geometry}                    
\usepackage{url}
\usepackage[colorlinks,linkcolor=blue,anchorcolor=blue,citecolor=blue]{hyperref}
\usepackage{color}
\usepackage{enumerate}
\usepackage{hhline}
\usepackage{graphicx, float}
\usepackage{subcaption}
\begin{document}

\newtheorem{theorem}{Theorem}[section]
\newtheorem{lemma}[theorem]{Lemma}
\newtheorem{conjecture}[theorem]{Conjecture}
\newtheorem{proposition}[theorem]{Proposition}
\newtheorem{corollary}[theorem]{Corollary}
\newtheorem{claim}[theorem]{Claim}
\newtheorem{example}[theorem]{Example}
\theoremstyle{definition}
\newtheorem{remark}[theorem]{Remark}
\newtheorem{definition}[theorem]{Definition}

\def\E{{\mathbb E}}
\def\F{{\mathbb F}}
\def\P{{\mathbb P}}
\def\R{{\mathbb R}}
\def\Z{{\mathbb Z}}
\def\C{{\mathbb C}}
\def\N{{\mathbb N}}
\def\O{{\mathbb O}}
\def\B{{\mathcal B}}
\def\G{{\mathcal G}}
\def\Hy{{\mathcal H}}
\def\M{{\mathcal M}}
\def\S{{\mathcal S}}
\def\T{{\mathcal T}}
\def\U{{\mathcal U}}

\newcommand{\rst}[1]{\ensuremath{{\mathbin\upharpoonright}\raise-.5ex\hbox{$#1$}}}


\title[On the number of directions formed by Cartesian products in $\mathbb{F}_{p^2}^2$]{On the number of directions formed by Cartesian products in $\mathbb{F}_{p^2}^2$}

 \author[A. Mohammadi] {Ali Mohammadi}

\email{ali.mohammadi.np@gmail.com}

\pagenumbering{arabic}

\begin{abstract}
We prove a lower bound on the number of directions determined by Cartesian products $A\times A$ in the affine plane over the finite field $\mathbb F_{p^2}$. Our lower bound holds for sets of size $p^{2/3}<|A|<p$, which are not contained in any affine copy of $\mathbb F_p$.

The proof combines a structural result of Li and Roche-Newton on the set of directions formed by Cartesian products with a lower bound of Fancsali, Sziklai and Tak\'{a}ts. A key step shows that, unless the set of directions exhibits closure properties forcing subfield structure, one obtains a direction for which an algebraic multiplicity parameter in the latter theorem can be made explicit.
\end{abstract}

\maketitle

\section{Introduction}
Let $q$ be a prime power. Given a point set $U\subset \mathbb{F}_q^2$, if $|U|>q$, a simple pigeonhole argument shows that it determines all possible directions. Indeed, for each fixed slope, there are exactly $q$ parallel lines of that slope, so if $U$ contains more than $q$ points, at least one such line must contain two points of $U$, yielding that direction.

For $|U|\le q$, the situation changes substantially, for example the point set may be collinear, in which case it determines only a single direction. This is the only degenerate case over prime fields, where sharp lower bounds are known (see e.g. \cite{DiBenedetto2021} and \cite{Szonyi1999}). Over general finite fields, point sets may be highly structured, particularly when concentrated on affine copies of subfields, and the number of determined directions can be far smaller than $|U|$. The extremal case $|U|=q$ was resolved by Ball~\cite{Ball2003}, who established the sharp lower bound $(q+1)/2$ outside the degenerate case of a line. For $|U|<q$, the problem was studied further by Fancsali, Sziklai and Tak\'ats~\cite{Fancsali2013}, who obtained general lower bounds via the standard R\'{e}dei polynomial method~\cite{Redei1970}. However, unlike the $|U|=q$ case, no concrete geometric characterisation is currently known for the point sets attaining the optimal lower bound.

In this note, we consider Cartesian products over the extension field
$\mathbb F_{p^2}$. For $A\subset\mathbb F_{p^2}$, the directions determined
by $A\times A$ are defined by
\[
D(A):=
\left\{
\frac{a_1-a_2}{a_3-a_4}:
a_1,a_2,a_3,a_4\in A,\ a_3\neq a_4
\right\}.
\]
Here the main obstruction is that $A$ may concentrate inside an affine copy of the
proper subfield $\mathbb F_p$. If $A\subset c\mathbb F_p+d$, then $D(A)\subseteq \mathbb F_p$, so $A\times A$ determines at most $p$ directions. This is far below the quadratic
scale when $|A|$ is close to $p$.

Our main result shows that, once this obstruction is excluded,
one recovers a  quadratic lower bound. The main result of \cite{Fancsali2013}, which we use as a preliminary result, could be refined using arguments of \cite{DiBenedetto2021} to recover the bound $|A|(|A|-1)$ which, as discussed in the latter, is sharp. It remains desirable to extend this lower bound to the range $|A|<p^{2/3}$, noting that the best currently known (sub-quadratic) bounds, in this range, can be obtained from
\cite{Macourt2020} and \cite{Petridis2019} based on methods from incidence geometry.
\begin{theorem}\label{thm:main}
Let $q=p^2$, for an odd prime $p$ and let $A\subset \mathbb F_q$.  
There exists an absolute constant $k>1$, such that if $k\cdot p^{2/3}<|A|<p$ and
 $A\not \subset c\mathbb F_p+d$, for every $c\in\mathbb F_q^\ast$ and $d\in\mathbb F_q$,
then $|D(A)|\geq (|A|^2+1)/2$.
\end{theorem}  

\section{Preliminary results}
Let $U\subset \mathbb{F}^2_q$ be a set of points. The associated R\'{e}dei polynomial is defined by
\[
R(X,Y):=\prod_{i=1}^{|U|} (X-a_iY+b_i), \quad (a_i, b_i)\in U.
\]
For a fixed direction $y\in\mathbb F_q$, the specialisation $R(X,y)$
encodes the intersections of $U$ with lines of slope $y$, since the roots of
$R(X,y)$ are precisely the intercepts $z$ such that the line
\[
\ell_{y,z}:=\{(x,yx-z):x\in\mathbb F_q\}
\]
meets $U$, and the multiplicity of a root equals the number of points of $U$
lying on that line. Since $R$ is monic, we may write $R(X,Y)Q(X,Y)=X^q+H(X,Y)$, with $\deg_X H<\deg_X R=|U|$.

\begin{definition}\label{def:sdef}
Let $U\subset\mathbb{F}^2_q$ and $D$ the set of directions formed by $U$. For $y\in D$, by $s(y)$ we denote the greatest power of $p$ such that each line $\ell$ of
direction $y$ meets $U$ in zero modulo $s(y)$ points. Let $s = \min_{y\in D}s(y)$.
\end{definition}

\begin{definition}\label{def:tdef}
Suppose $|D|\ge 2$, so that $H(X,y)$ is not constant. For each
$y\in D$, let $t(y)$ denote the maximal power of $p$ such that
\[
H(X,y)\in \mathbb F_q[X^{t(y)}]
\setminus
\mathbb F_q[X^{p\cdot t(y)}].
\]
Equivalently, there exists a polynomial $f_y(X)\notin \mathbb F_q[X^p]$ such that $H(X,y)=f_y(X)^{t(y)}$. Let $t=\min_{y\in D} t(y)$. If $H(X,y)\equiv a$ is constant (equivalently $D=\{y\}$), then we define $
t=t(y)=q$.
\end{definition}
For $y\in D$, by \cite[Propositions~9 and 11]{Fancsali2013}, we have
\[
R(X,y)\in \mathbb F_q[X^{s(y)}]\setminus \mathbb F_q[X^{p\cdot s(y)}]\quad\text{and}\quad H(X, y), Q(X, y)\in \mathbb F_q[X^{s(y)}],
\]
and $s(y)\le t(y)$.

We recall \cite[Theorem~17]{Fancsali2013}, which gives a lower bound on the number of directions formed by general point sets of size less than $q$. The relevant lower bound depends on the algebraic multiplicity parameter $t$, defined via an auxiliary polynomial construction (Definition~\ref{def:tdef}), rather than the more directly geometric parameter $s$. In general, it is not known whether these parameters coincide.
\begin{theorem}\label{thm:FancsaliFq}
Let $q$ be a prime power and $U\subset \mathbb{F}^2_q$ be an arbitrary set of points. Let $D$ denote the
set of directions determined by $U$. Let $s$ and $t$ be defined geometrically
and algebraically as in Definitions~\ref{def:sdef} and \ref{def:tdef}. Then one of the following
holds:

\[
1=s\le t<q\quad\text{and}\quad\frac{|U|-1}{t+1}+1
\le
|D|
\le
q;
\]

\[
1<s\le t<q\quad\text{and}\quad\frac{|U|-1}{t+1}+1
\le
|D|
\le
\frac{|U|-1}{s-1}-1.
\]
\end{theorem}

See \cite[Lemma~2.50]{TaoVu2006} for the following.
\begin{lemma}\label{lem:dilatesumsetcardinality}
Let $X \subset \mathbb{F}_q$ and let $r \in \mathbb{F}_q^\ast$. If $r \notin D(X)$, then for any nonempty subsets $X_1, X_2 \subseteq X$, we have $|X_1||X_2| = |X_1 \pm rX_2|$.
\end{lemma}

The following two variants of the Pl\"unnecke-Ruzsa inequality appear in \cite{KatzShen2008}.
\begin{lemma}\label{lem:plunnecke}
Let $X, B_1, \dots, B_k$ be nonempty subsets of an abelian group. Then
\[
|B_1 + \cdots + B_k|
\le
\frac{|X+B_1|\cdots |X+B_k|}{|X|^{k-1}}.
\]
\end{lemma}

\begin{lemma}\label{lem:plunneckerefined}
Let $X, B_1, \dots, B_k$ be nonempty subsets of an abelian group. For any $0<\varepsilon<1$, there exists a subset $X' \subseteq X$
with $|X'|>(1-\varepsilon)|X|$
such that
\[
|X' + B_1 + \cdots + B_k|
\ll_{\varepsilon,k}
\frac{|X+B_1|\cdots |X+B_k|}{|X|^{k-1}}.
\]
\end{lemma}

A key ingredient in the proof of our main result is the following structural characterisation of sets of slopes, implicit in the proof of the main theorem of \cite{LiRocheNewton2011}.
\begin{lemma}\label{lem:dirsubfieldcriterion}
Let $X\subset \mathbb F_q$, and suppose that
\[
X\,D(X)\subset D(X)\quad \text{and}\quad 1+D(X)\subset D(X).
\]
Then, $D(X)=F(X)$ where $F(X)$ denotes the subfield of $\mathbb F_q$ generated by $X$.
\end{lemma}

\section{Proof of Main Result}

\begin{claim}\label{lem:polygeomcriterion}
Let $q=p^2$, and let $A\subset \mathbb F_q$, with $|A|<p$. Let $y\in D(A)$.  Write
\[
R_y(X):=R(X,y)
=
\prod_{z\in yA-A}(X-z)^{r_y(z)},
\]
where $r_y(z):=|\{(a,b)\in A^2:ya-b=z\}|$. If $|yA-A|>p$, for some $y\in D(A)$, then $t(y)=s(y)=1$.
\end{claim}

\begin{proof}
For an arbitrary $y\in D(A)$, suppose that $P(X):=R_y(X)Q_y(X)\in \mathbb F_q[X^p]$. Thus, the multiplicity of a root $z\in yA-A$ of $P$ must be
divisible by $p$. But $z$ occurs in $R_y$ with multiplicity $r_y(z)$, where
\begin{equation}\label{eqn:ryzbound}
1\le r_y(z)\le |A|<p.
\end{equation}
Thus $Q_y$ must contain $X-z$ with multiplicity at least $p-r_y(z)$. Summing over all $z\in yA-A$, we get
\[
\deg Q_y
\ge
\sum_{z\in yA-A}(p-r_y(z))=
p|yA-A|-\sum_{z\in yA-A}r_y(z)
=
p|yA-A|-|A|^2.
\]
On the other hand, by the definition of $Q$, we get $\deg Q_y=p^2-|A|^2$. Hence if $|yA-A|> p$,
$R_y(X)Q_y(X)\notin \mathbb F_q[X^p]$.  By definition, $R_y(X)Q_y(X)=X^q+H(X,y)$, and so $H(X,y)\notin\mathbb F_q[X^p]$, which in turn forces $t(y)=1.$ The fact that $s(y)=1$ follows simply from \eqref{eqn:ryzbound}. 
\end{proof}

\begin{proof}[Proof of Theorem~\ref{thm:main}]
After replacing $A$ by an affine image
\[
A'=\frac{A-a_0}{a_1-a_0},
\qquad a_0,a_1\in A,\ a_0\ne a_1,
\]
we may assume that $0,1\in A$.

This normalisation preserves $|A|$, the condition $A\not\subset c\mathbb F_p+d$,
and $D(A)$. Since $0,1\in A$, we have $A\subset D(A)$. We show that either there exists $r\in D(A)$ such that $|rA-A|>p$, or else $D(A)=\mathbb F_q$.

Suppose, for contradiction, that
\begin{equation}\label{eqn:rA-ACon}
|rA-A|\le p
\qquad\text{for every }r\in D(A).
\end{equation}
We prove that $1+D(A)\subset D(A)
$ and $A D(A)\subset D(A)$. First suppose that $1+D(A)\not\subset D(A)$. Then there exists $r\in D(A)$ such that $1+r\notin D(A)$. Note that $D(A) = -D(A).$

By Lemmas~\ref{lem:dilatesumsetcardinality} and \ref{lem:plunneckerefined}, we have
for some subset $A_0\subseteq A$ with $|A_0|\gg |A|$, we get
\[
|A|^2 \ll |A_0+A+rA|
\leq
\frac{|A+A|\,|A+rA|}{|A|}\leq \frac{p^2}{|A|},
\]
and therefore $|A|\ll p^{2/3}$.
This contradicts the hypothesis $|A|>p^{2/3}$, up to constants. Hence $1+D(A)\subset D(A)$.

Next suppose that $A D(A)\not\subset D(A)$. Then there exist $a\in A$ and $r\in D(A)$ such that $ar\notin D(A)$. By Lemma~\ref{lem:dilatesumsetcardinality} and Lemma~\ref{lem:plunnecke}, and since $a\in A\subset D(A)$,
we obtain
\[
|A|^2 = |A+arA|
\le
\frac{|A+aA|\,|A+rA|}{|A|}
\le
\frac{p^2}{|A|}.
\]
Thus, $|A|\le p^{2/3}$
again contradicting the hypothesis. Therefore $A D(A)\subset D(A)$.

Invoking  Lemma~\ref{lem:dirsubfieldcriterion}, we have $D(A)=F(A)$, where $F(A)$ is the subfield of $\mathbb F_q$ generated by $A$. Suppose
$D(A)=\mathbb F_p$. Then, $A\subset D(A) =\mathbb F_p$. Undoing the affine normalisation, we get $A\subset c\mathbb F_p+d$, for some $c, d\in \F_q$, contradicting our assumption. If $D(A)=\mathbb F_q,$ we get $|D(A)|=q>|A|^2$. It remains to consider the case where \eqref{eqn:rA-ACon} fails. Then, there exists $y\in D(A)$ such that $|yA-A|>p$. By Claim~\ref{lem:polygeomcriterion}, we get $t=t(y)=s=s(y)=1$.
Applying Theorem~\ref{thm:FancsaliFq}, we obtain $|D(A)|\ge (|A|^2+1)/2,$ as required.
\end{proof}

\end{document}